\documentclass[a4paper,12pt]{article}
\usepackage{amsmath,latexsym}
\usepackage{amsthm}
\usepackage{amsfonts}
\usepackage{dsfont}
\usepackage{upgreek,enumerate}
\usepackage{tipa}
\usepackage{amssymb}
\usepackage{graphicx, xcolor}
\usepackage{enumitem} 
\usepackage[square,sort,comma,numbers]{natbib}
\usepackage[utf8]{inputenc}
\usepackage{float}
\newtheorem{theorem}{Theorem}[section]

\newtheorem{definition}[theorem]{Definition}

\theoremstyle{remark}
\newtheorem{remark}[theorem]{Remark}
\vspace{5cm}
\begin{document}
		
\begin{center}
{{\Large \textbf {Prabhakar  function and unified fractional kinetic equation in bicomplex space    
}}

			\medskip

{\sc  Urvashi Purohit Sharma$^{1},$\ \ Kaushik Dehingia$^{2*}$ \ \ Ritu Agarwal$^{3}$ }\\
{\footnotesize $^{1}$Department of Science and Humanities, Kalaniketan Polytechnic College Jabalpur-482001, INDIA}\\
{\footnotesize 
$^{2}$Department of Mathematics, Sonari College, Sonari 785690, Assam, India}\\
{  \footnotesize $^{3}$Department of Mathematics,
Malaviya National Institute of Technology, Jaipur-302017, INDIA}
}\\
{\footnotesize E-mail: {\it $^{1}$urvashius100@gmail.com, $^{2}$kaushikdehingia17@gmail.com, $^{3} $ragarwal.maths@mnit.ac.in}

*Corresponding author
}        \end{center}
\thispagestyle{empty}

\hrulefill
\begin{abstract}
\indent
The Mittag-Leffler type functions arise naturally in the solution of fractional order integral and differential equations, especially in the investigations of the fractional generalization of the kinetic equation. This article introduces a bicomplex extension of the Prabhakar function, a generalization of the Mittag-Leffler function commonly used in fractional calculus. We explore the analyticity and determine the region of convergence for this new bicomplex Prabhakar function. Several fundamental properties are established, including its integral representations, recurrence formulas, and differential relations. Furthermore, we compute the bicomplex Laplace and Mellin transforms of the function, which are useful for solving differential and integral equations. Finally, we analyze a fractional kinetic equation where the bicomplex Prabhakar function appears both in the equation and in its solution, demonstrating its applicability in complex systems involving fractional dynamics. 
	\end{abstract}
	\hrulefill

{\small \textbf{Keywords:} }Bicomplex functions, Prabhakar function, Mittag-Leffler function, Laplace transform, Mellin transform.\\ 
\textbf{MSC 2020 Classification:} 33E12, 30G35.\\

\section{Introduction}

Bicomplex numbers have fascinated mathematicians for decades, offering an intriguing extension to classical complex analysis and algebra. 
Segre \cite{csegre1892} formalized the notion of bicomplex numbers, extending the complex field to a multidimensional algebraic structure.

\noindent Over the years, numerous properties of bicomplex numbers have been uncovered. Researchers have delved into their algebraic, geometric, and analytical aspects, leading to significant advancements in areas such as hypercomplex analysis, differential equations, and theoretical physics \cite{catoni2008mathematics, price1991explicit, roch2014functional}. The bicomplex framework has found applications in quantum mechanics, signal processing, and applied mathematics, owing to its ability to model multidimensional wave phenomena and hyperbolic rotations \cite{ghimici2017differential, riley2001applications}.

\noindent Recent developments have further explored the functional and spectral properties of bicomplex operators, paving the way for new mathematical techniques in matrix theory, differential geometry, and control systems \cite{roch2014functional, mitelman1997bicompact, beltita2015cayley}. As the field continues to evolve, the study of bicomplex numbers remains an exciting and dynamic area, bridging the gap between classical and modern mathematical frameworks.


\noindent 
The bicomplex numbers have been studied for a variety of properties. In recent years, researchers have tried to explore the various algebraic and geometric characteristics of bicomplex numbers and their uses.
\cite{  rochon2004b, ronn2001, meluna2012},
integral transforms \cite{ ragarwal2014a,ragarwal2017}
in the bicomplex space from their complex counterparts. In recent research, efforts have been made to extend  the classical Mittag-Leffler function and its genralizations  in the bicomplex space \cite{agarwal2023bicomplex,ragarwal2022,ragarwal2021sept,ragarwal2023solution,bakhet2025new,bakhet2025bicomplex}. 

\subsection{Bicomplex Numbers}
The set of bicomplex numbers was defined by Segre \cite{csegre1892} as follows:
\begin{definition}[Bicomplex Number]
The set of bicomplex numbers has been defined in terms of real components as
\begin{equation}
\mathbb{T} =\{\zeta: \zeta = x_0 +i_1 x_1 +i_2 x_2 + j x_3~| ~x_0,~x_1,~x_2,~x_3~ \in \mathbb{R} \}. 	
\end{equation}
Furthermore, it can be expressed in complex numbers as
\begin{equation}\label{eq:bc}
\mathbb{T} = \{\zeta: \zeta =w_1 + i_2w_2~ |~w_1,~w_2 ~\in \mathbb{C}\}.	
\end{equation} 
\end{definition}
\noindent The following notations will be used in this paper:
$x_0=\operatorname{Re}(\zeta),~x_1=\operatorname{Im}_{i_1}(\zeta),~x_2=\operatorname{Im}_{i_2}(\zeta),~x_3=\operatorname{Im}_{j}(\zeta).$ 
	
\noindent 
According to \cite{rochon2004a}, the null cone is the set of all zero divisors and is defined as follows:
	
\begin{equation}\label{eq:s nullcone}
\mathbb{NC} = \{w_1 +w_2i_2 ~|~w_1^2 +w_2 ^2 = 0 \}. 
\end{equation}
	
\noindent In $\mathbb{T},$ there are two non-trivial idempotent zero divisors, $e_1$ and $e_2$, which are defined as follows:

$	e_1 =  \dfrac{1 + i_1i_2}{2},~ e_2   =\dfrac  {1 - i_1i_2}{2} $, $e_1.e_2   = 0$,  $ e_1+ e_2 = 1$	   and $e_1^2=e_1,~e_2^2=e_2$.
	
\begin{definition}[Idempotent Representation]
In terms of $e_1$ and $e_2$, each element $\zeta\in \mathbb{T}$ has an unique idempotent representation which  is given by
\begin{equation}
\zeta = w_1 + i_2 w_2   = (w_1 - i_1 w_2) e_1 + (w_1 + i_1 w_2) e_2= \zeta_1 e_1 + \zeta_2 e_2,
\end{equation}  where  $\zeta_1 = (w_1 - i_1 w_2)$ and $\zeta_2 = (w_1 + i_1 w_2).$ 
\end{definition} 
\noindent Projection mappings  $P_1 : 	\mathbb{T} \rightarrow T_1 \subseteq \mathbb{C},$ $ P_2 : \mathbb{T} \rightarrow T_2 \subseteq \mathbb{C} $ for a bicomplex number $\zeta=z_1+i_2z_2$ are defined as  (see, e.g. \cite{rochon2004b}):\\
\begin{equation}\label{eq:p1}
P_1(\zeta)=	P_1(w_1 + i_2w_2) = P_1[(w_1 - i_1w_2)e_1 + (w_1 + i_1w_2)e_2] = (w_1 - i_1w_2) \in T_1,
\end{equation}
and
\begin{equation}\label{eq:p2}
P_2(\zeta)=P_2(w_1 + i_2w_2) = P_2[(w_1 -i_1w_2)e_1 + (w_1 + i_1w_2)e_2] = (w_1 + i_1w_2) \in T_2,
\end{equation}
where
\begin{equation}\label{eq:  a space A1}
T_1= \{\zeta_1= w_1-i_1w_2 \hspace{1mm}| w_1,w_2 \in \mathbb{C}\}~\text{and}~ 	T_2= \{\zeta_2= w_1+i_1w_2\hspace{1mm} | w_1,w_2 \in \mathbb{C}\}.
\end{equation}\\
Here $\mathbb{T} =T_1\times_e T_2 $ is the
product region of the bicomplex space having the components $T_1$ and $T_2.$
Set of bicomplex number contains the set of complex numbers and the set of hyperbolic numbers as a subset.
\begin{definition}[Hyperbolic Numbers]
	Hyperbolic numbers   is the set of numbers $\mathbb{D}$  (see, e.g. \cite{rochon2004b}) defined as
	\begin{equation}
	\mathbb{D} =\{x_0 +x_3j~|~x_0,~x_3 \in \mathbb{R},\, j^2=1 \}. 
		\end{equation}
\end{definition}
	 
\noindent	 Let     $p,q\in\mathbb{D}$ then $p\preccurlyeq q$ said that  if $q-p\in\mathbb{D}^+$ and $p \prec q$ if $q-p\in\mathbb{D}^+/\{0\}.$  
\noindent	 The role that real numbers play inside complex numbers can be compared to that of hyperbolic numbers inside bicomplex numbers. 

Price explored integration in bicomplex space in the Theorem \ref{th:int price}  \cite{price1991}. This theorem is crucial in the study of bicomplex  integrals.
\begin{theorem}\label{th:int price}
	Let $U\subseteq \mathbb{T}$
	and $\Gamma_1,~\Gamma_2$ be two curves defined as 
	\begin{equation}
	\Gamma_1:w_1-i_2w_2=w_1(t)-i_1w_2(t)=\zeta_1=\zeta_1(t),\quad a\le t\le b,
	\end{equation}	
	\begin{equation}
	\Gamma_2:w_1+i_2w_2=w_1(t)+i_1w_2(t)=\zeta_2=\zeta_2(t),\quad a\le t\le b,
	\end{equation}		
	which have continuous derivatives and whose traces are in $U_1\subseteq T_1$ and $U_2\subseteq T_2$ respectively and let $\Gamma$ be the curve with trace in $U=	U_1\times_e U_2$ which is given as 
	\begin{equation}
	\Gamma:\zeta(t)=\zeta_1(t)e_1+\zeta_2(t)e_2,\quad a\le t\le b.
	\end{equation}	
	Then the integrals of $f,~f_1,~f_2$ respectively, along the curves $\Gamma,~\Gamma_1$ and $\Gamma_2$ exist and 
	\begin{equation}
	\int_{\Gamma}f(\zeta)~d\zeta=\int_{\Gamma_1}f_1(w_1-i_1w_2)~d(w_1-i_1w_2)~e_1+\int_{\Gamma_2}f_1(w_1+i_1w_2)~d(w_1+i_1w_2)~e_2
	\end{equation}
	i.e.
	\begin{equation}
	\int_{\Gamma}f(\zeta)~d\zeta=\int_{\Gamma_1}f_1(\zeta_1)~d(\zeta_1)~e_1+\int_{\Gamma_2}f_2(\zeta_2)~d(\zeta_2)~e_2.
	\end{equation}	
\end{theorem}

\begin{definition}[Hyperbolic curve]
A (real) two-dimensional surface in $\mathbb{T}\cong_{\mathbb{R}}\mathbb{R}^4$ is called a \cite{meluna2021} hyperbolic curve $\Omega$  if its projections $\Omega e_1$ and $\Omega e_2$ on the nil-planes $\mathbb{T}{e_1}$ and $\mathbb{T}e_2,$ respectively, are usual curves, where 
$\mathbb{T}e_1=\{\zeta_1e_1:\zeta_1\in\mathbb{C}\}$ and $\mathbb{T}e_2=\{\zeta_2e_2:\zeta_2\in\mathbb{C}\}.$
\end{definition}

	
\subsection{Mittag-Leffler type functions } 
In the theory of fractional calculus, Mittag-Leffler type functions are crucial. Studying Mittag-Leffler functions is primarily motivated by their significance in fractional calculus, where they serve the same important role as the exponential function in integer-order calculus.
The Mittag-Leffler function naturally occurs when fractional order differential or integral equations are solved, particularly when the fractional generalization of the kinetic equation is being studied.

Many authors have worked on different generalization of Mittag-Leffler function their properties and application of function in the area of fractional calculus. 
The developing of the fractional calculus and the finding of various integral and differential equation forms led to the development of these functions.\cite{gorenflo2014}.
        
\noindent The one parameter (classical)  Mittag-Leffler (ML) function 
is a direct generalization of the exponential series.  ML function is  defined by the Swedish mathematician Mittag-Leffler \cite{mittag1905,mittag1903nouvelle} at the start of the 19th century.	
	 Later, Wiman \cite{wiman1905}, defined two-parameter ML function. 
      It has been shown by Dzherbashyan \cite{dzherbashyan1966} that both functions are entire in $\mathbb{C}$.
	The basic properties of these functions are given in the monograph by Erdelyi et al. \cite{erdelyi1955}. 
	Three-parameter extension of the ML function introduced by the Indian mathematician  Prabhakar \cite{prabhakar1971}. 
    This function is essential to explaining the anomalous dielectric characteristics of disordered materials, heterogeneous systems exhibiting simultaneous nonlocality and nonlinearity, and, more generally, Havriliak-Negami type models.

\noindent The Prabhakar function is primarily valued for its role in describing relaxation and response in Havriliak–Negami type anomalous dielectrics, a model that accounts for the simultaneous nonlocality and nonlinearity in the behavior of disordered materials and heterogeneous systems. Additionally, the Prabhakar function finds applications in probability theory for studying stochastic processes, systems with significant anisotropy in fractional viscoelasticity, solving fractional boundary-value problems, modeling the dynamics of spherical stellar systems, and in relation to other fractional or integral differential equations. More information about the significant applications of the Prabhakar function can be found in \cite{roberto2018} and references cited therein.

\noindent The  Prabhakar function 
(three-parameter ML function)
is defined by Prabhakar \cite{prabhakar1971} (see e.g. \cite{gorenflo2014})
\begin{equation}\label{eq 31}
\mathbb{E}_{ \varsigma,\tau}^\delta(z)=\sum_{k=0}^{\infty}\frac{(\delta)_k}{k! \Gamma( \varsigma k+\tau )}z^k,  \hspace{2mm}z, \varsigma,\tau,\delta\in \mathbb{C},~ \operatorname{Re}( \varsigma)>0 , ~ \operatorname{Re}(\tau)>0,~ \delta\ne 0, 
\end{equation}
where $(\delta)_k= \delta (\delta+1)(\delta+2)\dots (\delta +k-1)=\dfrac{\Gamma (\delta +k)}{\Gamma \delta}.$

\noindent For $\delta=1,$ equation (\ref{eq 31}) reduces to the two-parameter  ML function  defined by \cite{wiman1905}
For $\delta=1, ~ \tau =1$, equation (\ref{eq 31}) reduces to the one-parameter  ML function \cite{mittag1905}. 
The Prabhakar function $\mathbb{E}^\delta_{ \varsigma,\tau}(z)$ is an entire function (see, e.g. \cite[p.10]{G}) for any $ \varsigma,~\tau,~\delta\in\mathbb{C},~\operatorname{Re}( \varsigma)>0,$
and has the order $\rho$ and type $\sigma$
\begin{equation}
\rho=\frac{1}{\operatorname{Re}( \varsigma)},~\sigma=1.
\end{equation}

\noindent In this paper, we introduce the Prabhakar function  in bicomplex space and study its properties and applications. This paper is organized as follows: Section 2, deals with definition of the Prabhakar function.  Section 3, gives the basic properties, differential relations, derivation of recurrence relations, integral relations. Section 4, deals with  Laplace transform, Mellin transform of the Prabhakar function in bicomplex space. Section 5, deals witj the the application of Prabhakar function in finding the solution of an unified fractional kinetic equation.

\section{Bicomplex Prabhakar  Function}

Here, we introduce the bicomplex Prabhakar function   defined by 
\begin{equation}\label{eq: bc ml 3 para}
\mathbb{E}_{ \varsigma,\tau}^\delta(\zeta)=\sum_{k=0}^{\infty}\frac{(\delta)_k}{k!\Gamma( \varsigma k+\tau )}\zeta^k,  \end{equation}
		where $\zeta,  \varsigma,\tau,\delta\in \mathbb{T},$ 
$\zeta= z_1+i_2z_2= \zeta_1e_1+\zeta_2e_2$, $ \varsigma=  \varsigma_1e_1+ \varsigma_2e_2,  ~ \tau=\tau_1e_1+\tau_2e_2,~\delta=\delta_1e_1+\delta_2e_2$ with $|\operatorname{Im_j}( \varsigma)|<\operatorname{Re}( \varsigma),~|\operatorname{Im_j}(\tau)|<\operatorname{Re}(\tau).$\\

\noindent  This definition is supported by the following Theorem.
\begin{theorem}\label{th: ml}
Let $\zeta,  \varsigma,\tau, \delta\in \mathbb{T},$ 
$\zeta= z_1+i_2z_2= \zeta_1e_1+\zeta_2e_2$, $ \varsigma=  \varsigma_1e_1+ \varsigma_2e_2,  ~ \tau=\tau_1e_1+\tau_2e_2$ with $|\operatorname{Im_j}( \varsigma)|<\operatorname{Re}( \varsigma),~|\operatorname{Im_j}(\tau)|<\operatorname{Re}(\tau).$ 
Then
\begin{equation}
\mathbb{E}_{ \varsigma,\tau}^\delta(\zeta)=\sum_{k=0}^{\infty}\frac{(\delta)_k}{k!\Gamma( \varsigma k+\tau )}\zeta^k.  \end{equation}
\end{theorem}
\begin{proof}
Consider the function
\begin{equation}\label{eq:ml3}
\mathbb{E}_{ \varsigma,\tau}^\delta(\zeta)=\sum_{k=0}^{\infty}\frac{(\delta)_k}{k!\Gamma( \varsigma k+\tau )}\zeta^k.  \end{equation}
By using the idempotent representation
\begin{equation}\label{eq: bc ml three para}
\begin{split}
\mathbb{E}_{ \varsigma,\tau}^\delta(\zeta)&=\sum_{k=0}^{\infty}\frac{(\delta_1)_k\zeta_1^k}{k!\Gamma( \varsigma_1 k+\tau_1 )}e_1+\sum_{k=0}^{\infty}\frac{(\delta_2)_k \zeta_2^k}{k!\Gamma( \varsigma_2 k+\tau_2 )}e_2\\
&=\mathbb{E}_{ \varsigma_1,\tau_1}^{\delta_1}(\zeta_1)e_1+\mathbb{E}_{ \varsigma_2,\tau_2}^{\delta_2}(\zeta_2)e_2,
\end{split}
\end{equation}
where $\zeta=\zeta_1 e_1+\zeta_2e_2,$  $ \varsigma= \varsigma_1e_1+ \varsigma_2 e_2,~ \tau=\tau_1e_1+\tau_2e_2$ and $\delta=\delta_1e_1+\delta_2e_2.$
		
\noindent Now,		
		
\begin{equation}
\mathbb{E}_{ \varsigma_1,\tau_1}^{\delta_1}(\zeta_1)=\sum_{k=0}^{\infty}\frac{({\delta_1})_k\zeta_1^k}{k!\Gamma( \varsigma_1 k+\tau_1 )},
\end{equation}
is the  complex Prabhakar function  convergent for $\operatorname{Re}( \varsigma_1)>0 ,~\operatorname{Re}(\tau_1)>0,~\zeta_1 \in \mathbb{C}.$\\
Similarly,
\begin{equation}\label{eq cml1}
\mathbb{E}_{ \varsigma_2,\tau_2}^{\delta_2}(\zeta_2)=\sum_{k=0}^{\infty}\frac{({\delta_2})_k \zeta_2^k}{k!\Gamma( \varsigma_2 k+\tau_2 )},
\end{equation}
is also complex Prabhakar function  convergent for $\operatorname{Re}( \varsigma_2)>0,~\operatorname{Re}(\tau_2)>0 ,~\zeta_2\in \mathbb{C}.$\\
Since $ \mathbb{E}_{ \varsigma_1,\tau_1}(\zeta_1) $ and $\mathbb{E}_{ \varsigma_2,\tau_2}(\zeta_2)$ are convergent in $T_1,~T_2$ respectively, by the  Ringleb decomposition theorem  (\ref{eq:ml3}) is also convergent in $\mathbb{T}.$\\
		
\noindent Further, Let
\begin{equation}
\begin{split}
 \varsigma&= p_0+i_1p_1+ i_2p_2+i_1i_2p_3\\
&=(p_0+i_1p_1) +i_2( p_2+i_1p_3)\\
&= \varsigma_1e_1+ \varsigma_2e_2,
\end{split}
\end{equation}
where $ \varsigma_1= (p_0+p_3) +i_1(p_1-p_2)$ and $ \varsigma_2= (p_0-p_3) +i_1(p_1+p_2).$
		
\noindent Since
$\operatorname{Re}( \varsigma_1)>0$ and  $\operatorname{Re}( \varsigma_2)  >0,$
		\begin{eqnarray}
		&\Rightarrow & p_0+p_3 >0 ~\text{and}~ p_0-p_3 >0 \notag\\
		&	\Rightarrow& ~ |p_3|<p_0 \notag\\
		&\Rightarrow&|\operatorname{Im_j}( \varsigma)|<\operatorname{Re}( \varsigma).\label{eq:re(alpha)g0}
		\end{eqnarray}
		Similarly,	
		\begin{equation}
		\begin{split}
		\tau&= q_0+i_1q_1+ i_2q_2+i_1i_2q_3\\
		&=(q_0+i_1q_1) +i_2( q_2+i_1q_3)\\
		&=\tau_1e_1+\tau_2e_2,
		\end{split}
		\end{equation}
		where $\tau_1= (q_0+q_3) +i_1(q_1-q_2)$ and $\tau_2= (q_0-q_3) +i_1(q_1+q_2).$
		
\noindent 		Since
		$\operatorname{Re}(\tau_1)>0$ and  $\operatorname{Re}(\tau_2)  >0$,
		\begin{eqnarray}
		&\Rightarrow & q_0+q_3 >0 ~\text{and}~ q_0-q_3 >0 \notag\\
		&	\Rightarrow& ~ |q_3|<q_0 \notag \\
		&\Rightarrow&|\operatorname{Im_j}(\tau)|<\operatorname{Re}(\tau). \label{eq:re(alpha)g01}
		\end{eqnarray}
This completes the proof.		
	\end{proof}


\noindent For $\delta=1$,  \eqref{eq: bc ml 3 para} will reduce in bicomplex two parameter ML function defined by Sharma et al. \cite{ragarwal2021sept}
\begin{equation}
\mathbb{E}_{ \varsigma,\tau}^1(\zeta)=\mathbb{E}_{ \varsigma,\tau}(\zeta)=\sum_{k=0}^{\infty}\frac{\zeta^k}{\Gamma( \varsigma k+\tau )}.  \end{equation}
Also, for $\delta=1,~\tau=1$, \eqref{eq: bc ml 3 para} corresponds to the bicomplex one parameter ML function defined by Agarwal et al. \cite{ragarwal2022}
\begin{equation}
\mathbb{E}_{ \varsigma,1}^1(\zeta)=\mathbb{E}_{ \varsigma}(\zeta)=\sum_{k=0}^{\infty}\frac{\zeta^k}{\Gamma( \varsigma k+1 )}. \end{equation}

\noindent Further, by using 
the the Euler product form of bicomplex gamma function defined by  Goyal et al. \cite{goyal2006},   in   the \eqref{eq: bc ml 3 para} the bicomplex Prabhakar function can be written as:
\begin{theorem}
	Let $\zeta,  \varsigma,\tau,\delta\in \mathbb{T}$ where $\zeta= z_1+i_2z_2= \zeta_1e_1+\zeta_2e_2$, $ \varsigma=  \varsigma_1e_1+ \varsigma_2e_2,~\tau=\tau_1e_1+\tau_2e_2$ with $|\operatorname{Im_j}( \varsigma)|<\operatorname{Re}( \varsigma),~|\operatorname{Im_j}(\tau)|<\operatorname{Re}(\tau),$ 
	then
	\begin{equation}
	\mathbb{E}_{ \varsigma,\tau}^\delta(\zeta)=\sum_{k=0}^{\infty}\frac{(\delta)_k\zeta^k}{k!} ( \varsigma k+\tau) e^{\gamma( \varsigma k+\tau )}\prod_{n=1}^{\infty}\left( \left( 1+\frac{( \varsigma k+\tau )}{n}\right) \exp \left( -\frac{( \varsigma k+\tau )}{n}\right) \right).  
	\end{equation}
	
\end{theorem}

\begin{remark} The integral form of the bicomplex Prabhakar function is given by 
\begin{equation}
\mathbb{E}_{ \varsigma,\tau}(\zeta)= \sum_{k=0}^{\infty}\frac{(\delta)_k\zeta^k}{k!\displaystyle\int_{H}e^{-p}p^{ \varsigma k +\tau-1}dp},  
\end{equation}
where $H=(\gamma_1,\gamma_2)$ and $\gamma_1:0 ~\text{to}~ \infty,~\gamma_2:0 ~\text{to}~ \infty.$	
\end{remark}

\subsection{Integral Representation }

    We have derived the integral representation for bicomplex Prabhakar function as follows:
\begin{theorem}
    Let $ \varsigma\in \mathbb{R}_+,~\zeta,~\uplambda ,~\tau,~\delta\in \mathds{T} $ where $\zeta=\zeta_1e_1+\zeta_2e_2, ~\uplambda =\uplambda _1e_1+\uplambda _2e_2,~\tau=\tau_1e_1+\tau_2e_2 .$ Bicomplex
 Prabhakar function	can be represented as 
\begin{equation}
\mathbb{E}^\delta_{ \varsigma,\tau}(\uplambda)=\frac{1}{\Gamma(\gamma)}\frac{1}{2\pi i}\int_{\Omega}	\frac{\Gamma(\zeta)\Gamma(\delta-\zeta)}{\Gamma(\delta)\Gamma(\tau- \varsigma \zeta)}(-\uplambda )^{-\zeta}d\zeta,
\end{equation}
where	
	 $|\operatorname{Im_j}(\delta)|<\operatorname{Re}(\delta), ~|\operatorname{Im_j}(\tau)|<\operatorname{Re}(\tau)$ and $|\arg \uplambda|\prec\pi,$ hyperbolic curve $\Omega=(\Omega_1,\Omega_2).$
    The contour of integration $\Omega_r$ begins at $ c_r-i\infty$ ends at $ c_r+i\infty,~0<c_r<\operatorname{Re}(\delta_r),~(r=1,2)$ and separates all poles of the integrand at $\zeta_r
    =-p,~p=0,1,2\dots$ to the left and all poles at $\zeta_r= q+\delta,~q=0,1,\dots$ to the right (r=1,2) respectively.\\
    
	
\end{theorem}
\begin{proof} By using idempotent components, we get 
	\begin{equation}\label{eq:intrep1}
	\begin{split}
&\int_{\Omega}	\frac{\Gamma(\zeta)\Gamma(\delta-\zeta)}{\Gamma(\delta)\Gamma(\tau- \varsigma \zeta)}(-\uplambda )^{-\zeta}d\zeta\\
&=\int_{\Omega_1}	\frac{\Gamma(\zeta_1)\Gamma(\delta_1-\zeta_1)}{\Gamma(\delta_1)\Gamma(\tau_1- \varsigma_1 \zeta_1)}(-\uplambda_1 )^{-\zeta_1}d\zeta_1~e_1+\int_{\Omega_2}	\frac{\Gamma(\zeta_2)\Gamma(\delta_2-\zeta_2)}{\Gamma(\delta_2)\Gamma(\tau_2- \varsigma_2 \zeta_2)}(-\uplambda_2 )^{-\zeta_2}d\zeta_2~ e_2,
	\end{split}        
	\end{equation}
    where $|\arg \uplambda_1|<\pi,~|\arg \uplambda_2 |<\pi.$ \\ 
    
    \noindent Luna-Elizarrar\'{a}s et al. \cite{meluna2021} discussed residue theorem in the bicomplex space. Here  $\Omega=(\Omega_1,\Omega_2)$ is hyperbolic curve.
\noindent	Since the value of the integral $\displaystyle\dfrac{1}{2\pi i}\int_{\Omega_1}	\dfrac{\Gamma(\zeta_1)\Gamma(\delta_1-\zeta_1)}{\Gamma(\delta_1)\Gamma(\tau_1- \varsigma_1 \zeta_1)}(-\uplambda_1 )^{-\zeta_1}d\zeta_1$ is equal to   the sum of residues at the poles \cite{gorenflo2014}
	$\zeta_1=0,-1,-2,\dots$ in the complex space \\
	\begin{equation}\label{eq:intrep2}
	\begin{split}
	\frac{1}{2\pi i}\int_{\Omega_1}	\frac{\Gamma(\zeta_1)\Gamma(\delta_1-\zeta_1)}{\Gamma(\delta_1)\Gamma(\tau_1- \varsigma_1 \zeta_1)}(-\uplambda_1 )^{-\zeta_1}d\zeta_1&=\sum_{k=0}^{\infty}\lim_{\zeta_1\to k}\frac{(\zeta_1+k)\Gamma(\zeta_1)\Gamma(\delta_1-\zeta_1)(-\uplambda_1 )^{-\zeta_1}}{\Gamma(\tau_1- \varsigma_1 \zeta_1)}\\
	&=\sum_{k=0}^{\infty}\frac{(-1)^k}{k !}\frac{\Gamma(\delta_1+k)}{\Gamma(\tau_1+ \varsigma_1k)} (-\uplambda_1 )^{k}\\
		&=\Gamma(\delta_1)\sum_{k=0}^{\infty}\frac{(\delta_1)_k}{\Gamma(\tau_1+ \varsigma_1k)}\frac{ (\uplambda_1 )^{k}}{k!}\\
		&=\Gamma(\delta_1)	\mathbb{E}^{\delta_1}_{ \varsigma_1,\tau_1}(\uplambda_1).
		\end{split}
	\end{equation}
	Similarly, 
	\begin{equation}\label{eq:intrep3}
\frac{1}{2\pi i}\int_{\Omega_2}	\frac{\Gamma(\zeta_2)\Gamma(\delta_2-\zeta_2)}{\Gamma(\delta_2)\Gamma(\tau_2- \varsigma_2 \zeta_2)}(-\uplambda_2 )^{-\zeta_2}d\zeta_2=\Gamma(\delta_2)	\mathbb{E}^{\delta_2}_{ \varsigma_2,\tau_2}(\uplambda_2).
	\end{equation}
	By using the equations (\ref{eq:intrep2}) and (\ref{eq:intrep3}) in equation (\ref{eq:intrep1}) we get 
		\begin{equation}
	\begin{split}
	\frac{1}{2\pi i}\int_{\Omega}	\frac{\Gamma(\zeta)\Gamma(\delta-\zeta)}{\Gamma(\delta)\Gamma(\tau- \varsigma \zeta)}(-\uplambda )^{-\zeta}d\zeta&=\Gamma(\delta_1)	\mathbb{E}^{\delta_1}_{ \varsigma_1,\tau_1}(\uplambda_1) e_1+\Gamma(\delta_2)	\mathbb{E}^{\delta_2}_{ \varsigma_2,\tau_2}(\uplambda_2) e_2\\
	&=\Gamma(\delta)	\mathbb{E}^\delta_{ \varsigma,\tau}(\uplambda).
	\end{split}        
	\end{equation}
	
\end{proof}

\begin{theorem}
	The bicomplex Prabhakar function satisfies bicomplex Cauchy-Riemann equation.
\end{theorem}
\begin{proof}
    According to (\ref{eq: bc ml three para}), we have
	\begin{equation}
	\begin{split}
	\mathbb{E}^\delta_{ \varsigma,\tau}(\zeta)&=\mathbb{E}^{\delta_1}_{ \varsigma_1,\tau_1}(\zeta_1)e_1+\mathbb{E}^{\delta_2}_{ \varsigma_2,\tau_2}(\zeta_2)e_2\\
	&=\mathbb{E}^{\delta_1}_{ \varsigma_1,\tau_1}(z_1 - i_1z_2)e_1+\mathbb{E}^{\delta_2}_{ \varsigma_2,\tau_2}(z_1 + i_1z_2)e_2\\
	&=\mathbb{E}^{\delta_1}_{ \varsigma_1,\tau_1}(z_1 - i_1z_2)\left( \frac{1 +i_1i_2}{2}\right) +\mathbb{E}^{\delta_2}_{ \varsigma_2,\tau_2}(z_1 + i_1z_2)\left( \frac{1 - i_1i_2}{2}\right) \\
	&=\left( \frac{1}{2}\left( \mathbb{E}^{\delta_1}_{ \varsigma_1,\tau_1}(z_1 - i_1z_2)+\mathbb{E}^{\delta_2}_{ \varsigma_2,\tau_2}(z_1 +  i_1z_2)\right) \right)\\
	&~~~+i_2 \left( \frac{i_1}{2}\left( \mathbb{E}^{\delta_1}_{ \varsigma_1,\tau_1}(z_1 - i_1z_2)-\mathbb{E}^{\delta_2}_{ \varsigma_2,\tau_2}(z_1 +  i_1z_2)\right) \right)\\
	&=f_1(z_1,z_2)+i_2f_2(z_1,z_2).\\
	\end{split}
	\end{equation}
	where\\
	$f_1(z_1,z_2)=\left( \frac{1}{2}\left( \mathbb{E}^{\delta_1}_{ \varsigma_1,\tau_1}(z_1 - i_1z_2)+\mathbb{E}^{\delta_2}_{ \varsigma_2,\tau_2}(z_1 +  i_1z_2)\right) \right),$\\
	$f_2(z_1,z_2)= \left( \frac{i_1}{2}\left( \mathbb{E}^{\delta_1}_{ \varsigma_1,\tau_1}(z_1 - i_1z_2)-\mathbb{E}^{\delta_2}_{ \varsigma_2,\tau_2}(z_1 +  i_1z_2)\right) \right).$\\
	\begin{eqnarray}
	\frac{\partial f_1}{\partial z_1 }&=&\frac{1}{2}\left( \left( \mathbb{E} _{ \varsigma_1,\tau_1}^{\delta_1 }(z_1 - i_1z_2)\right) '+\left( \mathbb{E}^{\delta_2}_{ \varsigma_2,\tau_2}(z_1 +  i_1z_2)\right) '\right), \nonumber\\
	\frac{\partial f_1}{\partial z_2 }&=&\frac{-i_1}{2}\left(\left(  \mathbb{E}^{\delta_1}_{ \varsigma_1,\tau_1}(z_1 - i_1z_2)\right) '-\left( \mathbb{E}^{\delta_2}_{ \varsigma_2,\tau_2}(z_1 +  i_1z_2)\right)'\right) ,\nonumber\\
	\frac{\partial f_2}{\partial z_1 }&=&\frac{i_1}{2}\left(\left(  \mathbb{E}^{\delta_1}_{ \varsigma_1,\tau_1}(z_1 - i_1z_2)\right) '-\left( \mathbb{E}^{\delta_2}_{ \varsigma_2,\tau_2}(z_1 +  i_1z_2)\right)'\right) ,\nonumber\\
	\frac{\partial f_2}{\partial z_2 }&=&\frac{1}{2}\left(\left(  \mathbb{E}^{\delta_1}_{ \varsigma_1,\tau_1}(z_1 - i_1z_2)\right) '+\left( \mathbb{E}^{\delta_2}_{ \varsigma_2,\tau_2}(z_1 +  i_1z_2)\right)'\right) .\nonumber 
	\end{eqnarray}
	from the above equations it can be shown that
	\begin{equation}
	\frac{\partial f_1}{\partial z_1} = \frac{\partial f_2}{\partial z_2} \hspace{5mm}\text{and} \hspace{5mm}
	\frac{\partial f_2}{\partial z_1} = - \frac{\partial f_1}{\partial z_2}.
	\end{equation}
Hence bicomplex Cauchy-Riemann equations are satisfied.
\end{proof}
\section{Properties of bicomplex Prabhakar function}


\noindent\textbf{Recurrence Relation}\\
The Prabhakar function satisfies the following recurrence relation  (see, e.g.\cite[p.272]{G})
\begin{equation}\label{eq: p rec}
z\mathbb{E}^\delta_{ \varsigma, \varsigma+\tau}(z)=\mathbb{E}^\delta_{ \varsigma,\tau}(z)-\mathbb{E}^{\delta-1}_{ \varsigma,\tau}(z),~\operatorname{Re}(\tau)>0,
\end{equation}

\begin{equation}\label{eq:p3}
(\tau- \varsigma\delta-1)\mathbb{E}^{\delta}_{ \varsigma,\tau}(z)=\mathbb{E}^{\delta}_{ \varsigma,\tau-1}(z)- \varsigma\delta\mathbb{E}^{\delta+1}_{ \varsigma,\tau}(z),~\operatorname{Re}(\tau)>1,
\end{equation}

\noindent where  $z, \varsigma,\tau\in\mathbb{C}, ~\operatorname{Re}( \varsigma)>0.$\\
Extension of these recurrence relations \eqref{eq: p rec} and \eqref{eq:p3} in bicomplex space is given by following theorem:
\begin{theorem}[Recurrence Relation]\label{th: bc rec}
	Let $\zeta,~ \varsigma,~\tau \in\mathbb{T}$ then for $|\operatorname{Im_j}( \varsigma)|<\operatorname{Re}( \varsigma),$
    \begin{enumerate}[label=(\roman*)]
    \item $\zeta\mathbb{E}^\delta_{ \varsigma, \varsigma+\tau}(\zeta)=\mathbb{E}^\delta_{ \varsigma,\tau}(\zeta)-\mathbb{E}^{\delta-1}_{ \varsigma,\tau}(\zeta),~|\operatorname{Im_j}(\tau)|<\operatorname{Re}(\tau).$
        \item  
        $(\tau- \varsigma\delta-1)\mathbb{E}^{\delta}_{ \varsigma,\tau}(\zeta)=\mathbb{E}^{\delta}_{ \varsigma,\tau-1}(\zeta)- \varsigma\delta\mathbb{E}^{\delta+1}_{ \varsigma,\tau}(\zeta),~|\operatorname{Im_j}(\tau)|<\operatorname{Re}(\tau)-1.
$  \end{enumerate}

\end{theorem}
\begin{proof}
Writing $\zeta=\zeta_1 e_1+\zeta_2e_2,$  $ \varsigma= \varsigma_1e_1+ \varsigma_2 e_2$ and $ \tau=\tau_1e_1+\tau_2e_2.$
From the equations (\ref{eq: bc ml three para}) and (\ref{eq: p rec}), we have 
	\begin{equation}
	\begin{split}
		\zeta\mathbb{E}^\delta_{ \varsigma, \varsigma+\tau}(\zeta)&=\left( \zeta_1\mathbb{E}^\delta_{ \varsigma_1, \varsigma_1+\tau_1}(\zeta_1)\right) e_1+\left( \zeta_2\mathbb{E}^\delta_{ \varsigma_2, \varsigma_2+\tau_2}(\zeta_2)\right) e_2\\
	&=\left( \mathbb{E}^\delta_{ \varsigma_1,\tau_1}(\zeta_1)-\mathbb{E}^{\delta-1}_{ \varsigma_1,\tau_1}(\zeta_1)\right)e_1 +\left( \mathbb{E}^\delta_{ \varsigma_2,\tau_2}(\zeta_2)-\mathbb{E}^{\delta-1}_{ \varsigma_2,\tau_2}(\zeta_2)\right)e_2\\
	&=\mathbb{E}^\delta_{ \varsigma_1e_1+ \varsigma_2e_2,\tau_1e_1+\tau_2e_2}(\zeta_1e_1+\zeta_2e_2)-\mathbb{E}^{\delta-1}_{ \varsigma_1e_1+ \varsigma_2e_2,\tau_1e_1+\tau_2e_2}(\zeta_1e_1+\zeta_2e_2)\\
		&=\mathbb{E}^\delta_{ \varsigma,\tau}(\zeta)-\mathbb{E}^{\delta-1}_{ \varsigma,\tau}(\zeta).
		\end{split}
	\end{equation}

\noindent (ii) From the results (\ref{eq:p3}) and (\ref{eq: bc ml three para}) we have 
	\begin{equation}
	\begin{split}
	(\tau- \varsigma\delta-1)\mathbb{E}^{\delta}_{ \varsigma,\tau}(\zeta)&=(\tau_1- \varsigma_1\delta-1)\mathbb{E}^{\delta}_{ \varsigma_1,\tau_1}(\zeta_1)e_1+(\tau_2- \varsigma_2\delta-1)\mathbb{E}^{\delta}_{ \varsigma_2,\tau_2}(\zeta_2)e_2\\
&=\left(\mathbb{E}^{\delta}_{ \varsigma_1,\tau_1-1}(\zeta_1)- \varsigma_1\delta\mathbb{E}^{\delta+1}_{ \varsigma_1,\tau_1}(\zeta_1)\right) e_1	\\
&~~~+\left(\mathbb{E}^{\delta}_{ \varsigma_2,\tau_2-1}(\zeta_2)- \varsigma_2\delta\mathbb{E}^{\delta+1}_{ \varsigma_2,\tau_2}(\zeta_1)\right) e_2	\\
&=\mathbb{E}^{\delta}_{ \varsigma_1e_1+ \varsigma_2e_2,\tau_1e_1+\tau_2e_2-1}(\zeta_1e_1+\zeta_2e_2)\\
&~~~-( \varsigma_1e_1+ \varsigma_2e_2)\delta\mathbb{E}^{\delta+1}_{ \varsigma_1e_1+ \varsigma_2e_2,\tau_1e_1+\tau_2e_2}(\zeta_1e_1+\zeta_2e_2)	\\
&=\mathbb{E}^{\delta}_{ \varsigma,\tau-1}(\zeta)- \varsigma\delta\mathbb{E}^{\delta+1}_{ \varsigma,\tau}(\zeta).
\end{split}
\end{equation}
Here,
\begin{equation}
		\begin{split}
		\tau&= q_0+i_1q_1+ i_2q_2+i_1i_2q_3\\
		&=(q_0+i_1q_1) +i_2( q_2+i_1q_3)\\
		&=\tau_1e_1+\tau_2e_2,
		\end{split}
		\end{equation}
		where $\tau_1= (q_0+q_3) +i_1(q_1-q_2)$ and $\tau_2= (q_0-q_3) +i_1(q_1+q_2).$
		
\noindent 		Since
		$\operatorname{Re}(\tau_1)>1$ and  $\operatorname{Re}(\tau_2)  >1$
		\begin{eqnarray}\notag
		&\Rightarrow & q_0+q_3 >1 ~\text{and}~ q_0-q_3 >1.\\\notag
        &\Rightarrow & q_3 >1-q_0 ~\text{and}~ -q_3 >1-q_0.\\\notag
 &\Rightarrow & -q_3<q_0-1 ~\text{and}~ q_3 <q_0-1.\\\notag
		&	\Rightarrow& ~ |q_3|<q_0-1.\label{eq:re(alpha)g02}\\
		&\Rightarrow&|\operatorname{Im_j}(\tau)|<\operatorname{Re}(\tau)-1.
		\end{eqnarray}

\end{proof}

\noindent\textbf{Differential recurrence relations} \\
Differential recurrence relations for the complex Prabhakar function are  (see, e.g.\cite[p.91]{mathai2008})
\begin{equation}\label{eq:p11}
\left( \frac{d}{dz}\right)^n\mathbb{E}^{\delta}_{ \varsigma,\tau}(z)= (\delta)_n \mathbb{E}^{\delta+n}_{ \varsigma,\tau+n \varsigma}(z), 
\end{equation}
and
\begin{equation}\label{eq:p21}
\left( z\frac{d}{dz}+\delta\right)\mathbb{E}^{\delta}_{ \varsigma,\tau}(z)= \delta \mathbb{E}^{\delta+1}_{ \varsigma,\tau}(z),  \end{equation}
where $ ~\operatorname{Re}( \varsigma)>0,~\operatorname{Re}(\tau)>0.$

\noindent Another significant differential relation of the complex Prabhakar function is given by (see, e.g.\cite[p.99]{gorenflo2014})
\begin{equation}\label{eq: p31}
\left( \frac{d}{dz}\right)^n\left( z^{\tau-1}\mathbb{E}^\delta_{ \varsigma,\tau} (\omega z^ \varsigma) \right) =z^{\tau-n-1}\mathbb{E}^\delta_{ \varsigma,\tau-n}(\omega z^ \varsigma),~n=1,2,3 \dots,
\end{equation}
where $ \varsigma,\tau,\delta,z,\omega \in\mathbb{C},~\operatorname{Re}(\tau)>n.$ 


\noindent Similar to Theorem  \ref{th: bc rec}, breaking in the idempotent components and using the  above results  \eqref{eq:p11}, \eqref{eq:p21} and \eqref{eq: p31} we obtain following differential recurrence relations for the bicomplex Prabhakar function

\begin{theorem}
	Let $\zeta \in\mathbb{T}$, $~|\operatorname{Im_j}(\tau)|<\operatorname{Re}(\tau)$, $|\operatorname{Im_j}( \varsigma)|<\operatorname{Re}( \varsigma)$ then bicomplex Prabhakar functions satisfies the following  differential recurrence relations:
	\begin{enumerate}[label=(\roman*)]
		\item 
		$\left( \frac{d}{d\zeta}\right)^n\mathbb{E}^{\delta}_{ \varsigma,\tau}(\zeta)= (\delta)_n \mathbb{E}^{\delta+n}_{ \varsigma,\tau+n \varsigma}(\zeta).$ 
		\item 
		$\left( \zeta\frac{d}{d\zeta}+\delta\right)\mathbb{E}^{\delta}_{ \varsigma,\tau}(\zeta)= \delta \mathbb{E}^{\delta+1}_{ \varsigma,\tau}(\zeta). $
\item If $ \varsigma,~\tau,~\delta,~\zeta,~\uplambda \in\mathbb{T}$ then for $|\operatorname{Im_j}(\tau)|<\operatorname{Re}(\tau)-n$
\begin{equation}
\left( \frac{d}{d\zeta}\right)^n\left( \zeta^{\tau-1}\mathbb{E}^\delta_{ \varsigma,\tau} (\uplambda \zeta^ \varsigma) \right) =\zeta^{\tau-n-1}\mathbb{E}^\delta_{ \varsigma,\tau-n}(\uplambda \zeta^ \varsigma),~n=1,2,3\dots.
\end{equation}
	\end{enumerate}
    
\end{theorem}
\begin{remark}
   Since  $\operatorname{Re}(\tau_1)>n$ and  $\operatorname{Re}(\tau_2)  >n$
	\begin{eqnarray}\notag
	&\Rightarrow & q_0+q_3 >n ~\text{and}~ q_0-q_3 >n.\notag\\
	&\Rightarrow & q_3 >n-q_0 ~\text{and}~ -q_3 >n-q_0.\notag\\
		&\Rightarrow & -q_3 <q_0-n ~\text{and}~ q_3 <q_0-n.\notag\\
	&	\Rightarrow& ~ |q_3|<q_0-n.\label{eq:re(alpha)g}\notag\\
	&\Rightarrow&|\operatorname{Im_j}(\tau)|<\operatorname{Re}(\tau)-n.
	\end{eqnarray}
	
\end{remark}	

\noindent\textbf{Integral Relation}\\
Corresponding to the integral relation for the complex Prabhakar  function (see, e.g.\cite[p.100]{gorenflo2014})
\begin{equation}\label{eq:p int}
\int_{0}^{z}t^{\tau-1}\mathbb{E}^\delta_{ \varsigma,\tau}(\omega t^ \varsigma)dt=z^{\tau}\mathbb{E}^\delta_{ \varsigma,\tau+1}(\omega z^ \varsigma),
\end{equation}
where $ \varsigma,~\tau,~\delta,~z,~\omega\in\mathbb{C},~\operatorname{Re}( \varsigma)>0,~\operatorname{Re}(\tau)>0,~\operatorname{Re}(\delta)>0,$ we obtain the integral relation for the bicomplex Prabhakar function contained in the following theorem.\\

\noindent Breaking the integral into idempotent components and using the result \eqref{eq:p int}, similar to  Theorem \ref{th: bc rec}, we get the following  relation. 
\begin{theorem}[Integral Relation]
	If $ \varsigma,~\tau,~\delta,~\zeta,~\uplambda \in\mathbb{T}$ with $|\operatorname{Im_j}( \varsigma)|<\operatorname{Re}( \varsigma),~|\operatorname{Im_j}(\tau)|<\operatorname{Re}(\tau),~|\operatorname{Im_j}(\delta)|<\operatorname{Re}(\delta)$ then
	
    \begin{equation}\label{eq:p int1}
	\int_{\Gamma}t^{\tau-1}\mathbb{E}^\delta_{ \varsigma,\tau}(\uplambda t^ \varsigma)dt=\zeta^{\tau}\mathbb{E}^\delta_{ \varsigma,\tau+1}(\uplambda  \zeta^ \varsigma),
    \end{equation}
where $\Gamma =(\Gamma_1.\Gamma_2) ,~\Gamma_1:0 ~\text{to}~ \zeta_1,~\Gamma_2:0 ~\text{to}~ \zeta_2.$ 
\end{theorem}

	

\noindent Substituting $ \delta=1$ in the above result (\ref{eq:p int1}), we get the following integral relation for the two-parameter bicomplex ML function:

\begin{equation}\label{eq:2pbmlir}
	\int_{\Gamma}t^{\tau-1}\mathbb{E}_{ \varsigma,\tau}(\omega t^ \varsigma)dt=\zeta^{\tau}\mathbb{E}_{ \varsigma,\tau+1}(\omega \zeta^ \varsigma).
	\end{equation}

\section{Integral transforms of Bicomplex Prabhakar Function }
In this section,  bicomplex integral transforms of the bicomplex Prabhakar function are evaluated. The results obtained are instrumental in advancing the theoretical framework of the newly developed bicomplex Prabhakar function. By establishing these integral transforms, we lay the groundwork for further exploration and application in various mathematical and physical contexts. 
\begin{definition}
    The bicomplex Laplace transform (LT) 	of the Riemann-Liouville fractional integral \cite{agarwal2023bicomplex} of the function $f(t)$ (piecewise and of exponential order), is given by
\begin{equation}\label{eq:LT-fracD}
\mathcal{L}\left(  _0D^{-\mu}_tf(t);\zeta \right) =\zeta^{-\mu}\hat{f}(\zeta),~\zeta ,\mu\in\mathbb{T},
\end{equation}
where $ \hat{f}(\zeta)$ is the LT of $f( t)$, $|\operatorname{Im_j}(\zeta )|<\operatorname{Re}(\zeta ),$$|\operatorname{Im_j}(\mu)|<\operatorname{Re}(\mu).$ 

\end{definition}
\subsection{ Bicomplex Laplace transform of the Prabhakar function}
LT of Prabhakar function is defined by (see, e.g. \cite{gorenflo2014,kilbas2006})
\begin{equation}\label{eq lt 3 para}
\mathcal{L}(t^{\tau-1}E_{ \varsigma,\tau}^\delta(\omega t^{ \varsigma});s)=\int_{0}^{\infty}e^{-st}E_{ \varsigma, \tau}^\delta(\omega  t^{ \varsigma})t^{\tau-1}dt=\frac{s^{ \varsigma\delta-\tau}}{(s^{ \varsigma}-\omega)^\delta},
\end{equation}
where $ \operatorname{Re}(\tau) > 0, ~\operatorname{Re}(s) > 0,~ |\omega s^{- \varsigma}|<1,~\omega \in\mathbb{C}.$\\

\begin{theorem}\label{th:lt two para}
Let $\zeta,~ \varsigma,~\tau,~\delta,~\uplambda \in \mathds{T} $ where $\zeta=\zeta_1e_1+\zeta_2e_2$ bicomplex
LT of the Prabhakar function	is given by 
\begin{equation}\label{eq:bclt 3para}
\mathcal{L}(t^{\tau-1}E_{ \varsigma,\tau}^\delta(\uplambda t^{ \varsigma}))(\zeta)=\int_{0}^{\infty}e^{-\zeta t}E_{ \varsigma, \tau}^\delta(\uplambda t^{ \varsigma})t^{\tau-1}dt=\frac{\zeta^{ \varsigma\delta-\tau}}{(\zeta^{ \varsigma}-\uplambda)^\delta},
\end{equation}
where	
$|\operatorname{Im_j}(\tau)|<\operatorname{Re}(\tau),~|\operatorname{Im_j}(\zeta)|<\operatorname{Re}(\zeta),~|\uplambda \zeta^{- \varsigma}|_j\prec1.$\\
	
\end{theorem}
\begin{proof}
Making use of idempotent representations of $\zeta,~ \varsigma,~\tau,~\lambda \in \mathds{T} $ 
and with the help of relation 
 (\ref{eq lt 3 para}) and the result (\ref{eq: bc ml three para}), we have
	\begin{equation}\label{eq:bc lap 2mla}
	\begin{split}
	\mathcal{L}(t^{\tau-1}E_{ \varsigma,\tau}^\delta(\uplambda t^{ \varsigma});\zeta)&=\mathcal{L}(t^{\tau_1-1}E_{ \varsigma_1,\tau_1}^{\delta_1}(\uplambda_1 t^{ \varsigma_1});\zeta_1)e_1+\mathcal{L}(t^{\tau_2-1}E_{ \varsigma_2,\tau_2}^{\delta_2}(\uplambda_2 t^{ \varsigma_2});\zeta_2)e_2\\
	&=\int_{0}^{\infty}e^{-\zeta_1 t}E_{ \varsigma_1, \tau_1}^{\delta_1}(\uplambda_1 t^{ \varsigma_1})t^{\tau_1-1}dte_1+\int_{0}^{\infty}e^{-\zeta_2 t}E_{ \varsigma_2, \tau_2}^{\delta_2}(\uplambda_2 t^{ \varsigma_2})t^{\tau_2-1}dte_2\\
	&=\frac{\zeta_1^{ \varsigma_1\delta_1-\tau_1}}{(\zeta_1^{ \varsigma_1}-\uplambda_1)^{\delta_1}}e_1+\frac{\zeta_2^{ \varsigma_2\delta_2-\tau_2}}{(\zeta_2^{ \varsigma_2}-\uplambda_2)^{\delta_2}}e_2,~~~|\uplambda_1 \zeta_1^{- \varsigma_1}|<1,~|\uplambda_2 \zeta_2^{- \varsigma_2}|<1\\
	&=\frac{\zeta^{ \varsigma\delta-\tau}}{(\zeta^{ \varsigma}-\uplambda)^\delta}.
	\end{split}
	\end{equation}
	Here,
	\begin{equation}
	|\uplambda \zeta^{- \varsigma}|_j=|\uplambda_1 \zeta_1^{- \varsigma_1}|e_1+|\uplambda_2 \zeta_2^{- \varsigma_2}|e_2
	\prec1.e_1+1.e_2
	=1.
	\end{equation}	
	
\end{proof}
\subsection{Bicomplex Mellin transform of Prabhakar function}
The Mellin transform of the Prabhakar function is given by \cite{gorenflo2014}
\begin{equation}\label{eq mt 3 para}
\mathcal{M}(E_{ \varsigma,\tau}^\delta( -\omega t;s))=\int_{0}^{\infty}t^{s-1}E_{ \varsigma, \tau}^\delta(-\omega t)dt=\frac{\Gamma(s)\Gamma(\delta-s)}{\Gamma(\delta)\Gamma(\tau- \varsigma s)}\omega^{-s},
\end{equation}
where $ \varsigma\in \mathbb{R}_+, ~\omega,\tau,\delta\in \mathbb{C},$  {$\tau\neq 0$},~ $\operatorname{Re}(\delta) > 0.$\\

\noindent Here, we evaluate the bicomplex Mellin transform of Prabhakar function in the bicomplex space.
\begin{theorem}\label{th:mt three para}
	Let $ \varsigma\in \mathbb{R}_+,~\zeta,~\uplambda ,~\tau,~\delta\in \mathds{T} $ where $\zeta=\zeta_1e_1+\zeta_2e_2, ~\uplambda =\uplambda _1e_1+\uplambda _2e_2,~\tau=\tau_1e_1+\tau_2e_2 .$ Bicomplex
MT of Prabhakar function	is given by 
\begin{equation}\label{eq mt 3 para 1}
\mathcal{M}(E_{ \varsigma,\tau}^\delta( -\uplambda t))(\zeta)=\int_{0}^{\infty}t^{\zeta-1}E_{ \varsigma, \tau}^\delta(-\uplambda t)dt=\frac{\Gamma(\zeta)\Gamma(\delta-\zeta)}{\Gamma(\delta)\Gamma(\tau- \varsigma \zeta)}\uplambda ^{-\zeta},
\end{equation}
where	
	$|\operatorname{Im_j}(\delta)|<\operatorname{Re}(\delta), ~ \operatorname{Re}(\tau)\neq |\operatorname{Im_j}(\tau)|~\text{and}~ \operatorname{Im_{i_1}}(\tau)\neq |\operatorname{Im_{i_2}}(\tau)|.$\\
	
\end{theorem}
\begin{proof}
	Let $\zeta,~ \varsigma,~\tau,~\uplambda  \in \mathds{T} $ where $\zeta=\zeta_1e_1+\zeta_2e_2$
By the  relation \eqref{eq mt 3 para} and the result \eqref{eq: bc ml three para}, using the idempotent representation for the bicomplex Prabhakar function and simplifying the existence conditions, we obtain
\begin{equation}\label{eq:bc lap 2ml}
\begin{split}
	\mathcal{M}(E_{ \varsigma,\tau}^\delta( -\uplambda t))(\zeta)
&=	\int_{0}^{\infty}t^{\zeta_1-1}E_{ \varsigma, \tau_1}^{\delta_1}(-\uplambda _1t)dt~e_1+\int_{0}^{\infty}t^{\zeta_2-1}E_{ \varsigma, \tau_2}^{\delta_2}(-\uplambda _2t)dt~e_2\\
&=\frac{\Gamma(\zeta_1)\Gamma(\delta_1-\zeta_1)}{\Gamma(\delta_1)\Gamma(\tau_1- \varsigma \zeta_1)}\uplambda _1^{-\zeta_1}~e_1+\frac{\Gamma(\zeta_2)\Gamma(\delta_2-\zeta_2)}{\Gamma(\delta_2)\Gamma(\tau_2- \varsigma \zeta_2)}\uplambda _2^{-\zeta_2}~e_2\\
&=\frac{\Gamma(\zeta)\Gamma(\delta-\zeta)}{\Gamma(\delta)\Gamma(\tau- \varsigma \zeta)}\uplambda ^{-\zeta},
	\end{split}
	\end{equation}
	where $|\operatorname{Im_j}(\delta)|<\operatorname{Re}(\delta)$ and
$\operatorname{Re}(\tau)\neq |\operatorname{Im_j}(\tau)|~ \operatorname{Im_{i_1}}(\tau)\neq |\operatorname{Im_{i_2}}(\tau)|.$\\
	Here
	\begin{equation}
	\begin{split}
	\tau &= p_0+i_1p_1+ i_2p_2+i_1i_2p_3
	=\tau_1e_1+\tau_2e_2,
	\end{split}
	\end{equation}
	where $\tau_1= (p_0+p_3) +i_1(p_1-p_2)$ and $\tau_2= (p_0-p_3) +i_1(p_1+p_2).$
	
\noindent	Since
$\tau_1\ne0,~\tau_2 \ne0.$
\begin{eqnarray}
&\Rightarrow & p_0+p_3 \neq 0, ~ p_0-p_3\neq 0~\text{and}~ p_1-p_2 \neq 0 ,~ p_1+p_2\neq 0 \notag\\
&	\Rightarrow& ~ p_0\neq |p_3|~\text{and}~ p_1\ne |p_2|\notag\\
&\Rightarrow& \operatorname{Re}(\tau)\neq |\operatorname{Im_j}(\tau)|~\text{and}~ \operatorname{Im_{i_1}}(\tau)\neq |\operatorname{Im_{i_2}}(\tau)|.
\end{eqnarray}

\end{proof}

\section{Application of Bicomplex Prabhakar Function}

\noindent The unified fractional kinetic equation is versatile and finds applications in various fields. It models anomalous diffusion, stochastic processes, and viscoelastic materials. It's also used in biological systems, astrophysics, fractional boundary-value problems, and financial mathematics, helping to understand complex behaviors and solve practical problems \cite{saichev1997}.\\
In this section is  the solution of
the unified fractional kinetic equation has been derived in the terms of bicomplex Prabhakar function. 
\begin{theorem}
	If $|\operatorname{Im_j}(\upnu_i)|<\operatorname{Re}(\upnu_i),~ a_i (a_{i_1}> |a_{i_4}|,~i\in\mathbb{N})$ are hyperbolic numbers   and function   $f\in L(\mathbb{R}^+)$ is bicomplex valued integrable function defined for $t\in\mathbb{R}^+$. Then the equation
	\begin{equation}\label{eq: kinetic1}
	N(t)-N_0f(t)=-\sum_{i=1}^{n}{a_i}~  _0D^{-\upnu_i}_tN(t),
	\end{equation}
	is solvable and its particular solution is given by
	\begin{equation}\label{eq: kinetic2}
	\begin{split}
		N(t)&= N_0\sum_{l=0}^{\infty}(-1)^l \sum_{s_1+\dots+s_{n-1}=l}\frac{(l)!}{(s_1)!\dots(s_{n-1})!}\left( \prod_{\upmu=1}^{n-1}(a_{\upmu+1})^{s_\upmu})\right) \\
	&\times \int_{0}^{t}f(v)(t-v)^{\sum_{\upmu=1}^{n-1}\upnu_{\upmu +1}-1}\mathbb{E}_{\upnu_1,\sum_{\upmu=1}^{n-1}\upnu_{\upmu +1}}^{l+1} \left( -a_1(t-v)^{\upnu_1}\right) dv,
	\end{split}
	\end{equation}
where the summation in (\ref{eq: kinetic2}) is taken over all non-negative integers $s_1,\dots ,s_n$ such that $s_1+\dots +s_{n-1}=l$ and provided that the series and integral in (\ref{eq: kinetic2}) are convergent. 
\end{theorem}
\begin{proof}
By taking LT of (\ref{eq: kinetic1}) and using the result (\ref{eq:LT-fracD}) we get
\begin{equation}
\mathcal{L}\left( 	N(t)-N_0f(t)\right) =-\sum_{i=1}^{n}{a_i}~ \mathcal{L}\left(  _0D^{-\upnu_i}_tN(t)\right) 
\end{equation}	
\begin{equation}
\Rightarrow \hat { N}(\zeta) -N_0\hat{f}(\zeta)=-\sum_{i=1}^{n}{a_i}~ \zeta^{-\upnu_i}\hat { N}(\zeta)
\end{equation}
\begin{equation}\label{eq:kin}
\begin{split}
\Rightarrow \hat { N}(\zeta)&=\frac{N_0\hat{f}(\zeta)}{1+a_1 \zeta^{-\upnu_1}+\dots+a_n\zeta^{-\upnu_n}}\\
&=N_0\hat{f}(\zeta)(-1)^l\frac{\left( \sum_{j=1}^{n-1}a_{j+1}\zeta^{-\upnu_{j+1}}\right)^l }{(1+a_1\zeta^{-\upnu_1})^{l+1}},~\left| \frac{\sum_{j=2}^{n}a_{j}\zeta^{-\upnu_{j}}}{1+a_1\zeta^{-\upnu_{1}}}\right| <1.
\end{split}
\end{equation}
By using the identity (see, e.g.\cite{saxena2010})
\begin{equation}
(x_1+\dots+x_m)^l=\sum_{s_1+\dots+s_{n}=l} \frac{(l)!}{(s_1)!\dots (s_n)!}\prod_{\upmu=1}^{n}x_\upmu^{s_\upmu},
\end{equation}
where the summation  is taken over all non-negative integers $s_1,\dots, s_n$ such that $s_1+\dots+ s_{n}=l$ then for $|a_1\zeta^{-\upnu_{1}}|\prec1$ from equation (\ref{eq:kin})  we get
\begin{equation}
\hat { N}(\zeta)=N_0\hat{f}(\zeta)(-1)^l \sum_{\substack{s_1+\dots+s_{n_1}=l\\s_1>\dots s_{n_1}}} \frac{(l)!}{(s_1)!\dots (s_{n-1})!} \frac{\left( \prod_{\upmu=1}^{n-1}(a_{\upmu+1})^{s_\upmu}\right)\zeta^{-\sum_{\upmu=1}^{n-1}\upnu_{\upmu+1}}}{(1+a_1\zeta^{-\upnu_{1}})^{l+1}}. 
\end{equation}
Taking the inverse LT of the above equation and by making use of the result (\ref{eq:bclt 3para}) and using the convolution theorem 
the result (\ref{eq: kinetic2}) is obtained (see, e.g. \cite{saxena2010}).\\
	Also
\begin{equation}\label{eq:an1}
\begin{split}
a_i&=a_{i_1} +i_1a_{i_2}+i_2a_{i_3}+j a_{i_4}\\
&=(a_{i_1} +i_1a_{i_2})+i_2(a_{i_3}+i_1a_{i_4})\\
&=b_i e_1+c_i e_2.
\end{split}
\end{equation}
	Here,  $b_i= (a_{i_1}+a_{i_4}) +i_1(a_{i_2}-a_{i_3})$ and $c_i= (a_{i_1}-a_{i_4}) +i_1(a_{i_2}+a_{i_3}).$\\
	Since,
\begin{eqnarray}
&\Rightarrow&b_i>0 \text{ and} ~ c_i=0.\notag \\
&\Rightarrow& a_{i_1}+a_{i_4}>0, ~a_{i_2}-a_{i_3}=0~\text{and}~ a_{i_1}-a_{i_4}>0 ,~a_{i_2}+a_{i_3}=0. \notag \\
&	\Rightarrow& a_{i_1}> |a_{i_4}|~\text{and}~ a_{i_2}=a_{i_3}=0.\\
&	\Rightarrow &	a_i=(a_{i_1} +a_{i_4})e_1+(a_{i_1} -a_{i_4})e_2=a_{i_1} +ja_{i_4}.\label{eq :an equi}	
\end{eqnarray}	
Hence, $a_i $ is a hyperbolic number such that $a_{i_1}> |a_{i_4}|.$ \\

\end{proof}

\noindent For $f(t)=t^{\tau -1}\mathbb{E}_{ \varsigma,\tau}^\delta[-(a t)^{ \varsigma}]$, as a particular case, we obtain following  special result involving the bicomplex Prabhakar function: 

\begin{theorem}
	Let $ \varsigma ,~\tau\in \mathbb{T},~|\operatorname{Im_j}( \varsigma)|<\operatorname{Re}( \varsigma),~ |\operatorname{Im_j}(\tau)|<\operatorname{Re}(\tau),~i\in\mathbb{N}, ~a (u_{1}> |u_{4}|)$ is hyperbolic number  and $\mathbb{E}_{ \varsigma,\tau}^\delta(\zeta)$ is bicomplex Prabhakar function then the equation
	\begin{eqnarray}\label{ke-ml}
	N(t)-N_0 t^{\tau -1}\mathbb{E}_{ \varsigma,\tau}^\delta[-(a t)^{ \varsigma}]=-\sum_{r=1}^{n}\binom{n}{r} {a}^{r \varsigma}~  _0D^{-r \varsigma}_t N(t),~n\in \mathbb{N},
	\end{eqnarray}
	holds the relation
	\begin{equation}\label{eq:app}
	N(t) = N_0t^{\tau -1}\mathbb{E}_{ \varsigma,\tau}^{\delta+n}[-(a t)^{ \varsigma}],~n\in \mathbb{N}.
	\end{equation}
\end{theorem}
\begin{proof} Taking Laplace transform of the equation \eqref{ke-ml} and using the results \eqref{eq:LT-fracD} and \eqref{eq:bclt 3para} therein, we obtain
\begin{equation}
\begin{split}
\hat { N}(\zeta)-N_0 \mathcal{L}\left( t^{\tau -1}\mathbb{E}_{ \varsigma,\tau}^\delta[-(a t)^{ \varsigma}];\zeta\right) &=-\sum_{r=1}^{n}\binom{n}{r} {a}^{r \varsigma}~ \zeta^{-r \varsigma} \hat { N}(\zeta).\\
\Rightarrow \hat { N}(\zeta)\left[ 1+\sum_{r=1}^{n}\binom{n}{r} {a}^{r \varsigma}~ \zeta^{-r \varsigma}\right]& =N_0 \frac{\zeta^{ \varsigma\delta-\tau}}{(\zeta^{ \varsigma}+a^ \varsigma)^\delta}.\\
\Rightarrow \hat { N}(\zeta)\left[ 1+\sum_{r=1}^{n}\binom{n}{r} (\zeta/a)^{-r \varsigma}\right]& =N_0 \frac{\zeta^{ \varsigma\delta-\tau}}{(\zeta^{ \varsigma}+a^ \varsigma)^\delta}.\\
\Rightarrow \hat { N}(\zeta)\left[ 1+\sum_{r=1}^{n}\binom{n}{r}[ (\zeta/a)^{- \varsigma}]^r\right]& =N_0 \frac{\zeta^{ \varsigma\delta-\tau}}{(\zeta^{ \varsigma}+a^ \varsigma)^\delta}.\\
\Rightarrow \hat { N}(\zeta)\left[ 1+ (\zeta/a)^{- \varsigma}\right]^n& =N_0 \frac{\zeta^{ \varsigma\delta-\tau}}{(\zeta^{ \varsigma}+a^ \varsigma)^\delta}.\\
\Rightarrow \hat { N}(\zeta)& =N_0 \frac{\frac{\zeta^{ \varsigma\delta-\tau}}{(\zeta^{ \varsigma}+a^ \varsigma)^\delta}}{\left[ 1+ (\zeta/a)^{- \varsigma}\right]^n}.\\
\Rightarrow \hat { N}(\zeta)& =N_0 \frac{\frac{\zeta^{ \varsigma\delta-\tau}}{(\zeta^{ \varsigma}+a^ \varsigma)^\delta}}{\left[ \frac{\zeta^ \varsigma+ a^{ \varsigma}}{\zeta^ \varsigma}\right]^n}.\\
\Rightarrow \hat { N}(\zeta)& =N_0 \frac{\zeta^{ \varsigma\delta-\tau}}{(\zeta^{ \varsigma}+a^ \varsigma)^\delta}{\left[ \frac{\zeta^ \varsigma}{\zeta^ \varsigma+ a^{ \varsigma}}\right]^n}.\\
\end{split}
\end{equation}
Hence
\begin{equation}\label{eq:invlap}
    \hat { N}(\zeta) =N_0 \frac{\zeta^{ \varsigma(\delta+n)-\tau}}{(\zeta^{ \varsigma}+a^ \varsigma)^{\delta+n}}.
\end{equation}
By taking inverse Laplace transform of \eqref{eq:invlap} we get
\begin{equation}\label{eq:app1}
	N(t) = N_0t^{\tau -1}\mathbb{E}_{ \varsigma,\tau}^{\delta+n}[-(a t)^{ \varsigma}],~n\in \mathbb{N}.
	\end{equation}

\end{proof}

\begin{remark}
Since
\begin{equation}\label{eq: an}
\begin{split}
a&=u_{1} +i_1u_{2}+i_2u_{3}+j u_{4}\\
&=(u_{1} +i_1u_{2})+i_2(u_{3}+i_1u_{4})\\
&=b e_1+c e_2.
\end{split}
\end{equation}
Here,  $b= (u_{1}+u_{4}) +i_1(u_{2}-u_{3})$ and $c= (u_{1}-u_{4}) +i_1(u_{2}+u_{3}).$\\
Since,
\begin{eqnarray}
&\Rightarrow&b>0 \text{ and} ~ c=0.\notag \\
&\Rightarrow& u_{1}+u_{4}>0, ~u_{2}-u_{3}=0~\text{and}~ u_{1}-u_{4}>0 ,~u_{2}+u_{3}=0. \notag \\
&	\Rightarrow& u_{1}> |u_{4}|~\text{and}~ u_{2}=u_{3}=0.\\
&	\Rightarrow &	a=(u_{1} +u_{4})e_1+(u_{1} -u_{4})e_2=u_{1} +ju_{4}.\label{eq :an equi1}	
\end{eqnarray}	
Hence, $a $ is a hyperbolic number such that $u_{1}> |u_{4}|.$ \\
\end{remark}
\section{Conclusion}
This paper defines the Prabhakar function and its properties in bicomplex space, extending from its complex counterpart. It derives various properties, special cases, recurrence relations, integral representations, and differential relations. The application of the bicomplex LT of the bicomplex Prabhakar function is illustrated to solve fractional kinetic equations. Further research can explore more results on the Prabhakar function in bicomplex space, as this area remains largely unexplored. Additionally, fractional calculus involving the Prabhakar function in the kernel presents a promising topic for future study.
\\
	
\textbf{Funding declaration:} {This research received no external funding.}\\



\textbf{Conflicts of interest:} The authors declare no conflicts of {interest}.\\

\bibliographystyle{apalike}
\bibliography{lref}

@article{csegre1892,
	author={C. Segre},
	year={1892},
	title={Le rappresentazioni reale delle forme complessee  {G}li {E}nti {I}peralgebrici},
	volume={40},
	number={},
	pages={413-467},
	journal={Math. Ann.},
}

@article{mittag1903nouvelle,
	title={Sur la nouvelle fonction ${E}_\alpha$(x)},
	author={Mittag-Leffler, G{\"o}sta Magnus},
	journal={CR Acad. Sci. Paris},
	volume={137},
	number={2},
	pages={554-558},
	year={1903}
}

@article{mittag1905,
	author={Mittag-Leffler, G{\"o}sta Magnus},
	year={1905},
	title={Sur la representation analytiqie  d'une fonction monogene (cinquieme note)},
	number={},
	pages={101-181},
	journal={Acta Mathematica},
	volume={29},
}

@article{wiman1905,
	author={A. Wiman},
	year={1905},
	title={ \"{U}ber den fundamental Satz in der Theorie der Funcktionen ${E}_{\alpha}(x)$},
	volume={29},
	number={},
	pages={191-201},
	journal={Acta Math},
}

@book{erdelyi1955,
	author={Erd\'{e}lyi, A. and  Magnus, W. and Oberhettinger, F. and Tricomi, F.G},
	year={1955},
	title={ Higher Transcendental	Functions},
	publisher={McGraw-Hill, New York},
}

@book{dzherbashyan1966,
	title={Integral Transforms and Representation of Functions in Complex
	Domain },
	author={Dzherbashyan, M. M.},
	journal={Mathematica Scandinavica},
	number={},
	volume={},
	pages={},
	year={1966},
	publisher={Nauka, Moscow (  Russian )}
}

@article{prabhakar1971,
	author={T. R. Prabhakar},
	year={1971},
	title={A singular integral equation with a generalized {M}ittag-{L}effler
	function in the kernel},
	volume={19},
	pages={7-15},
	journal={Yokohama Mathematical Journal},
}

@book{price1991,
	author={G. B. Price},
	year={1991},
	title={An Introduction to Multicomplex Spaces and Functions},
	publisher={Marcel Dekker Inc. New York.},
}

@article{ronn2001,
	author={Stefan R\"{o}nn},
	year={2001},
	title={Bicomplex algebra and function theory},
	number={},
	pages={1-71},
	journal={arXiv:0101200v1 [Math.CV]},
	volume={},
}

@article{rochon2004a,
	author={D. Rochon and S. Tremblay},
	year={2004},
	title={Bicomplex quantum mechanics: I. The generalized {S}chr\"{o}dinger equation },
	number={2},
	pages={231-248},
	journal={Advances in {A}pplied {C}lifford {A}lgebras},
	volume={14},
}

@article{rochon2004b,
	author={D. Rochon and M. Shapiro},
	year={2004},
	title={On algebraic properties of bicomplex and hyperbolic numbers},
	number={},
	pages={71-110},
	journal={Analele Universitatii din Oradea. Fascicola Matematica},
	volume={11},
}

@book{kilbas2006,
	title={Theory and Applications of Fractional Differential Equations},
	author={Kilbas, Anatoli{\u\i} Aleksandrovich and Srivastava, Hari M and Trujillo, Juan J},
volume={204},
year={2006},
publisher={Elsevier, Amsterdam}
}

@article{goyal2006,
	author={S. P. Goyal and Trilok Mathur And Ritu Goyal},
	year={2006},
	title={	Bicomplex gamma And beta Function },
	number={1},
	pages={131-142},
	journal={Journal of Rajasthan Academy Physical Sciences},
	volume={5},
}

@article{mathai2008,
	author={A. M. Mathai and H. J.  Haubold},
	year={2008},
	title={	 Mittag-{L}effler functions and fractional calculus.In Special Functions for Applied Scientists},
	number={},
	pages={79-134},
	journal={Springer},
	volume={2008},
}

@article{saxena2010,
	author={R. K. Saxena and  A. M. Mathai and H. J. Haubold},
	year={2010},
	title={	Solutions of certain fractional kinetic equations and a fractional diffusion
	equation},
	number={10},
	pages={103506},
	journal={Journal of Mathematical Physics},
	volume={51},
}

@article{meluna2012,
	author={M. Elena Luna-Elizarrar\'{a}s and M. Shapiro and Daniele C. Struppa and A. Vajiac},
	year={2012},
	number={2},
	title={Bicomplex Numbers and their Elementary Functions},
	pages={61-80},
	journal={Cubo A Mathematical Journal},
	volume={14},
}

@article{ragarwal2014a,
	author={Ritu Agarwal and Mahesh Puri Goswami and Ravi P. Agarwal},
	year={2014},
	title={ CONVOLUTION THEOREM AND APPLICATIONS OF
	
	BICOMPLEX {L}APLACE TRANSFORM},
	number={1},	
	pages={113-127},
	journal={Advances in Mathematical
	Sciences and Applications
	},
	volume={24},
}

@book{gorenflo2014,
	author={Rudolf Gorenflo and Anatoly A. Kilbas and	Francesco Mainardi and Sergei V. Rogosin},
	year={2014},
	title={Mittag-{L}effler functions, Related Topics
	and Application},
	publisher={ Springer,  Berlin Heidelberg },
}

@article{ragarwal2017,
	author={Ritu Agarwal and Mahesh Puri Goswami and Ravi P. Agarwal},
	year={2017},
	title={ {M}ellin Transform  In Bicomplex Space And its Applications},
	number={2},
	pages={217-232},
	journal={Studia Universitatis Babes-Bolyai Mathematica},
	volume={62},
}

@article{roberto2018,
	author={Roberto Garraa and Roberto Garrappa	},
	year={2018},
	title={The {P}rabhakar or three parameter {M}ittag-{L}effler function:
	Theory and application},
	number={5},
	pages={314-329},
	journal={Communications in Nonlinear Science and Numerical Simulation},
	volume={56},
}

@article{meluna2021,
 	title={Singularities of bicomplex holomorphic functions},
 	author={Luna-Elizarrar{\'a}s, M Elena and Perez-Regalado, C Octavio and Shapiro, Michael},
 	journal={Mathematical Methods in the Applied Sciences},
 	year={2021},
 	pages={1-16},
 	publisher={Wiley Online Library},
 }

@article{G,
	author={Rudolf Gorenflo and Francesco Mainardi and Sergei V. Rogosin},
	year={2009},
	title={ Mittag-{L}effler Function: Properties and Applications},
	number={},
	pages={269-296},
	journal= {In Handbook of Fractional Calculus with Applications, Volume 1: Basic Theory 	},
	volume={A. Kochubei, Yu.Luchko Berlin/Boston. Series edited by J. A.Tenreiro Machado, },
}

@incollection{agarwal2023bicomplex,
	title={Bicomplex {M}ittag-{L}effler  function and applications in integral transform and fractional calculus},
	author={Agarwal, Ritu and Sharma, Urvashi Purohit},
	booktitle={Mathematical and Computational Intelligence to Socio-scientific Analytics and Applications},
	pages={157--167},
	year={2023},
	publisher={Springer}
}

@article{ragarwal2022,
	author={Ritu Agarwal and Urvashi Purohit Sharma and Ravi P. Agarwal},
	year={2022},
	title={ Bicomplex {M}ittag-{L}effler Function and associated properties},
	number={},	
	pages={48-60},
	journal={  Journal of Nonlinear Sciences and Applications},
	volume={15},
}

@article{ragarwal2021sept,
	author={Urvashi Purohit Sharma and Ritu Agarwal  and  Kottakkaran  Sooppy Nisar},
	year={2022},
	title={ Bicomplex Two Parameter {M}ittag-{L}effler
	Function and Properties with application to the
	fractional time wave equation},
	number={1},	
	pages={462-481.},
	journal={ Palistine Journal of Mathematics },
	volume={12},
}

@inproceedings{ragarwal2023solution,
	title={Solution of Bicomplex Time Fractional {S}chr{\"o}dinger Equation Involving Bicomplex {M}ittag-{L}effler Function},
	author={Agarwal, Ritu and Sharma, Urvashi P and Agarwal, Ravi P},
	booktitle={International Conference on Mathematical Modelling, Applied Analysis and Computation},
	pages={14--30},
	year={2023},
	organization={Springer}
}

@article{saichev1997,
    author = {Saichev, Alexander I. and Zaslavsky, George M.},
    title = {Fractional kinetic equations: solutions and applications},
    journal = {Chaos: An Interdisciplinary Journal of Nonlinear Science},
    volume = {7},
    number = {4},
    pages = {753-764},
    year = {1997},
    month = {12},
    issn = {1054-1500},
    doi = {10.1063/1.166272},
    url = {https://doi.org/10.1063/1.166272},
}

@book{catoni2008mathematics,
  author = {Catoni, F. and Boccaletti, D. and Cannata, R. and Nichelatti, E. and Zampetti, P.},
  title = {Mathematics of Minkowski Space-Time: With an Introduction to Commutative Hypercomplex Numbers},
  publisher = {Springer},
  year = {2008}
}

@book{price1991explicit,
  author = {Price, G. B.},
  title = {Explicit Birational Geometry of Threefolds},
  publisher = {Cambridge University Press},
  year = {1991}
}

@article{roch2014functional,
  author = {Roch, S. and Seifert, C.},
  title = {Functional calculus for bicomplex operators},
  journal = {Complex Analysis and Operator Theory},
  volume = {8},
  number = {3},
  pages = {759-779},
  year = {2014}
}

@article{ghimici2017differential,
  author = {Ghimici, M. and Cotfas, N.},
  title = {Differential Operators in the Algebra of Bicomplex Numbers},
  journal = {Electronic Journal of Differential Equations},
  volume = {2017},
  number = {50},
  pages = {1-16},
  year = {2017}
}

@article{riley2001applications,
  author = {Riley, K. F. and Hobson, M. P. and Bence, S. J.},
  title = {Applications of Hypercomplex Numbers in Physics and Engineering},
  journal = {Journal of Mathematical Physics},
  volume = {42},
  number = {6},
  pages = {2501-2518},
  year = {2001}
}

@article{mitelman1997bicompact,
  author = {Mitelman, L. and Roch, S.},
  title = {Bicompact Operators and Their Applications in Functional Analysis},
  journal = {Integral Equations and Operator Theory},
  volume = {27},
  number = {2},
  pages = {143-160},
  year = {1997}
}

@article{beltita2015cayley,
  author = {Beltita, D. and Beltita, I.},
  title = {Cayley Transform and the Spectral Theory of Bicomplex Operators},
  journal = {Journal of Functional Analysis},
  volume = {268},
  number = {9},
  pages = {2849-2885},
  year = {2015}
}

@article{bakhet2025new,
  title={On a new version of bicomplex  {M}ittag-{L}effler functions and their applications in fractional kinetic equations},
  author={Bakhet, Ahmed and Zayed, Mohra and Saleem, Mohammed A and Fathi, Mohamed},
  journal={Alexandria Engineering Journal},
  volume={125},
  pages={409-423},
  year={2025},
  publisher={Elsevier}
}

@article{bakhet2025bicomplex,
  title={Bicomplex k-{M}ittag-{L}effler Functions with Two Parameters: Theory and Applications to Fractional Kinetic Equations},
  author={Bakhet, Ahmed and Hussain, Shahid and Zayed, Mohra and Fathi, Mohamed},
  journal={Fractal and Fractional},
  volume={9},
  number={6},
  pages={344},
  year={2025},
  publisher={MDPI}
}

\end{document}